\documentclass{article}

\usepackage{amsmath,amsfonts,amssymb,amsthm}
\usepackage{graphics,epsfig,graphicx,graphics}
\usepackage{color}
\usepackage{pict2e,epic}

\usepackage{mathrsfs}
\usepackage{latexsym}
\usepackage{xcolor}

\usepackage{tikz}
\usetikzlibrary{patterns}
\usepackage{hyperref}

\usepackage{vmargin}
\setmarginsrb{2cm}{1cm}{2cm}{2cm}{1cm}{1cm}{1cm}{1cm}

\def\Om{\Omega}

\def\K{\mathcal{K}}
\def\R{\mathbb{R}}

\numberwithin{equation}{section}
\numberwithin{figure}{section}

\newtheorem{Theorem}{Theorem}[section]

\newtheorem{Proposition}[Theorem]{Proposition}

\newtheorem{Remark}[Theorem]{Remark}

\title{Elasticae and inradius}
\author{Antoine Henrot \and Othmane Mounjid}
\date{\today}

\usepackage{amsmath}

\begin{document}
\maketitle

\begin{abstract}
The elastic energy of a planar convex body is defined by $E(\Om)=\frac 12\,\int_{\partial\Om} k^2(s)\,ds$
where $k(s)$ is the curvature of the boundary. In this paper we are interested in the minimization problem
of $E(\Om)$ with a constraint on the inradius of $\Om$. By contrast with all the other minimization problems
involving this elastic energy (with a perimeter, area, diameter or circumradius constraints) for which the
solution is always the disk, we prove here that the solution of this minimization problem is not the disk and we completely characterize
it in terms of elementary functions.
\end{abstract}




\section{Introduction}
Following L. Euler, we define the elastic energy of a regular planar convex body $\Omega$ (a planar convex compact set) by
the formula
$$
E(\Omega)=\displaystyle \frac 12 \int_{\partial \Omega} k^2(s) \ ds
$$
where $k$ is the curvature and $s$ is the arc length. We will denote by $\K$ the class of regular ($C^{1,1}$) bounded planar convex
bodies. Several recent works involving this elastic energy appeared during the two last years,
see \cite{BHT}, \cite{BH}, \cite{FKN}. In particular, the authors were interested in finding sharp inequalities between $E(\Om)$,
the area $A(\Om)$ or the perimeter $P(\Om)$. More generally, we can consider several minimization problems for $E(\Om)$ when
putting different geometric constraints on $\Om$.  We will use the notation $D(\Om)$ for the diameter, $R(\Om)$ for the circumradius
(radius of the smallest disk
containing $\Om$) and $r(\Om)$ for the inradius (radius of the largest disk contained in $\Om$). Let us remark that $E$ scales
as $t^{-1}$ under a dilation: $E(t\Om)=E(\Om)/t$, thus it will always be equivalent to consider the minimization problem with
an equality constraint or with an inequality constraint.
Let us recall what is known (or what is an easy consequence of what is known) in
that context.

\smallskip\noindent
{\bf Minimization with a perimeter constraint}:\\
The solution of $\min\{E(\Om), \Om\in\K, P(\Om)\leq p\}$ is the disk of perimeter $p$.\\
Indeed by using the Cauchy-Schwarz inequality we deduce
\begin{equation}\label{minper}
2\pi= \int_{\partial \Omega} k \ ds
\leq \left(\int_{\partial \Omega} k^2 \ ds \right)^{1/2} \left(P(\Omega)\right)^{1/2}=\sqrt{2p}\sqrt{E(\Om)},
\end{equation}
with equality only in the case of a disk.
This minimal property can also be written by homogeneity
\begin{equation}\label{int1}
    \forall \Om\in\K,\quad E(\Om) P(\Om) \geq 2\pi^2\,.
\end{equation}
Let us remark that this minimal property can be extended in a straightforward way to any regular set $\Om$: by filling
the holes we decrease both perimeter and elastic energy, so it is enough to consider only simply connected sets for
which the previous proof works as well.

\smallskip\noindent
{\bf Minimization with an area constraint}: \\
The solution of $\min\{E(\Om), \Om\in\K, A(\Om)\leq a\}$ is the disk of area $a$.\\
For convex domains, this result has been proved by M.E. Gage in \cite{Gage}. It has been recently extended to simply connected
domains by D. Bucur and the first author in \cite{BH} and by V. Ferone, B. Kawohl and C. Nitsch in \cite{FKN}. It is trivially
wrong if we remove the assumption of simple-connectedness: consider a ring $\{R_1\leq |x|\leq R_2\}$ with $R_1,R_2 \to +\infty$.
This minimal property can also be written by homogeneity
\begin{equation}\label{int2}
    \forall \Om\in\K,\quad E(\Om)^2 A(\Om) \geq \pi^2\,.
\end{equation}

\smallskip\noindent
{\bf Minimization with a diameter constraint}: \\
The solution of $\min\{E(\Om), \Om\in\K, D(\Om)\leq d\}$ is the disk of diameter $d$.\\
Indeed, it is well known that for any plane convex domain $\Om$, the inequality $P(\Om)\leq \pi D(\Om)$ holds true (with equality
for any set of constant width), see \cite{SA}. Thus, combining with (\ref{int1}) yields
\begin{equation}\label{int3}
D(\Om) E(\Om) \geq \frac{P(\Om) E(\Om)}{\pi}\,\geq 2\pi
\end{equation}
with equality only for the disk. This result still holds for any regular planar set, since taking the convex hull does not change
the diameter while it decreases the elastic energy.

\smallskip\noindent
{\bf Minimization with a circumradius constraint}: \\
The solution of $\min\{E(\Om), \Om\in\K, R(\Om)\leq R\}$ is the disk of radius $R$.\\
Indeed, it is well known that for any plane convex domain $\Om$, the inequality $P(\Om)\leq 2\pi R(\Om)$ holds true (with equality
for any set of constant width), see \cite{SA}. Thus, combining with (\ref{int1}) yields
\begin{equation}\label{int4}
R(\Om) E(\Om) \geq \frac{P(\Om) E(\Om)}{2\pi}\,\geq \pi
\end{equation}
with equality only for the disk. This result still holds for any regular planar set, since taking the convex hull does not change
the circumradius while it decreases the elastic energy.

\smallskip\noindent
The last problem we want to consider is the minimization problem with an inradius constraint. Without loss of generality and because
the homogeneity property, we can fix the inradius to be one. We will denote by $\mathbb{D}$ the unit disk. Thus we want to study:
\begin{equation}\label{intr}
    \min\{E(\Om), \Om\in\K, r(\Om)\leq 1\}
\end{equation}
Since we work with bounded convex domains, the strip $\{0\leq y\leq 1\}$ (which would have zero elastic energy)
is not admissible.
It turns out that this problem is more difficult and surprisingly, we will discover that the optimal domain {\bf is not} the disk.
More precisely, the main result of this paper is
\begin{Theorem}\label{mainTh}
For any convex domain $\Om\in\K$ with inradius $r(\Om)$, the following inequality holds
\begin{equation}\label{theo1}
    E(\Om) r(\Om) \geq 2 \left( \int_0^{\frac{\pi}{2}} \sqrt{\cos t}\,dt\right)^2\,.
\end{equation}
Equality in (\ref{theo1}) is obtained for the convex domain $\Om^*$ symmetric with respect to $x=0$
and defined on $x\geq 0$ by, see Figure \ref{optimal_curve}:

\begin{figure}[h]\label{optimal_curve}
\begin{center}
\scalebox{.5}{\includegraphics{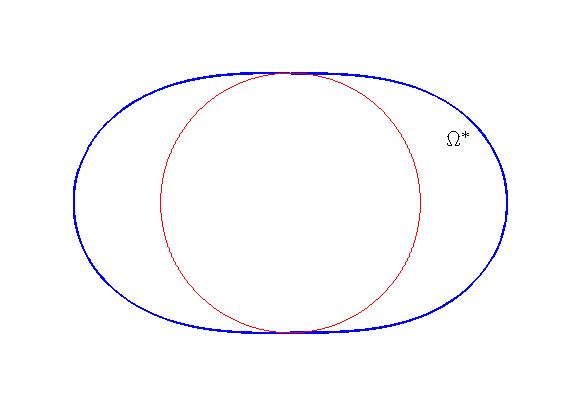}}
\caption{The optimal convex domain}
\end{center}
\end{figure}

\begin{equation}\label{optdom}
    \left\lbrace
    \begin{array}{ccc}
      x(s) & = & \frac{2}{a}\,\sqrt{\sin \theta(s)} \\
      y(s) & = & -1 + \frac{1}{a}\,\int_0^{\theta(s)}\sqrt{\sin u} du
    \end{array}\right.
  \quad s\in [0,L_0]
\end{equation}
where
\begin{equation}\label{optdom2}
    a=\int_0^{\frac{\pi}{2}} \sqrt{\cos t} dt,\quad L_0=\frac{2}{a}\,\int_0^{\frac{\pi}{2}} \frac{dt}{\sqrt{\cos t}},\quad \theta \mbox{ solution of the ODE }
     \left\lbrace
    \begin{array}{c}
      \theta'(s)=a\sqrt{\sin\theta(s)} \\
      \theta(0)=0
    \end{array}\right.
\end{equation}
Moreover, the elastic energy of the optimal domain is $E(\Omega^*)=2 a^2$.
\end{Theorem}
\begin{Remark}
The optimal domain $\Omega^*$ described in Theorem \ref{mainTh} is clearly not unique. Indeed any "stadium-like"
domain obtained by inserting two equal horizontal segments at north and south poles of $\Omega^*$ will provide
another $C^{1,1}$ optimal domain.
\end{Remark}
\begin{Remark}
Let us remark that it can be seen directly that $\Omega^*$ has a lower elastic energy than the unit
disk $\mathbb{D}$.
Indeed, using Cauchy-Schwarz inequality for $a$ yields
$$E(\Omega^*)=2\left(\int_0^{\frac{\pi}{2}} \sqrt{\cos t} dt\right)^2 < 2 \frac{\pi}{2} \,
\int_0^{\frac{\pi}{2}} \cos t dt=\pi=E(\mathbb{D}).$$
\end{Remark}
The proof of Theorem \ref{mainTh} will follow the classical method of calculus of variations:
existence, regularity and
use of the optimality conditions to derive the optimal set which is described above. More precisely,
fixing two consecutive contact points $A=(0,-1)$ and $B=(\sin 2\alpha, -\cos 2\alpha)$
of a domain with its inscribed (unit) disk, we are
led to study the sub-problem of minimizing the elastic energy of a (convex) arc $\gamma$ joining $A$ and $B$
with tangents $\tau_A=(1,0)$ and $\tau_B=(\cos 2\alpha, \sin 2\alpha)$.
\vspace{0.4 cm}

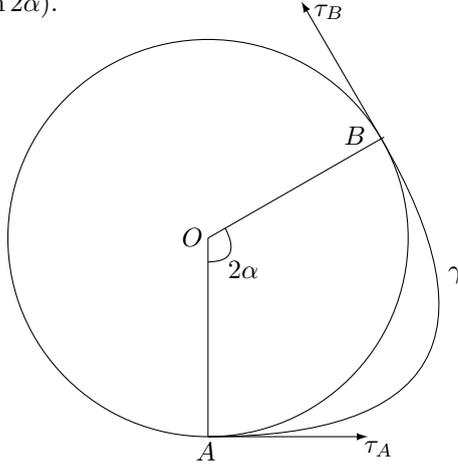
\begin{figure}[h]
\begin{center}
\setlength{\unitlength}{0.75pt}
\begin{picture}(200,200)(-120,-120)
\put(0,0){\line(0,-1){100}}
\put(0,0){\line(1.73,1){88}}
\put(0,0){\circle{200}}
\put(0,-100){\vector(1,0){80}}
\put(86.6,50){\vector(-1,1.73){40}}
\qbezier(0,-100)(173,-100)(86.6,50)
\qbezier(0,-12)(17.3,-12)(8.6,5)
\put(-6,-112){$A$}
\put(78,-108){$\tau_A$}
\put(68,47){$B$}
\put(53,112){$\tau_B$}
\put(-13,-4){$O$}
\put(120,-20){$\gamma$}
\put(10,-20){$2\alpha$}
\end{picture}
\caption{Looking for an optimal arc $\gamma$}
\end{center}
\end{figure}

If we denote by $L$ the length of the arc $\gamma$ (which is an unknown), this minimization problem
$(\mathcal{P}_\alpha)$ reads
\begin{equation}\label{pbpalpha}
(\mathcal{P}_\alpha)\quad \min\{E(\gamma), \gamma \mbox{ convex }, \gamma(0)=A, \gamma'(0)=\tau_A,
\gamma(L)=B, \gamma'(L)=\tau_B\}.
\end{equation}
In Section \ref{section2}, we prove existence of a minimizer for $(\mathcal{P}_\alpha)$. It turns out that it is
easy when $\alpha <\pi/2$ while it is much more complicated when $\alpha=\pi/2$ which corresponds to two
contact points diametrically opposite. Then Section \ref{section3} is devoted to write the optimality conditions.
We use it to prove the $C^2$ regularity of the optimal set. The good parametrization consists in working
with $\theta(s)$ the angle of the tangent with the horizontal axis. Due to the convexity asumption
(which can be seen as a constraint), the optimality conditions reads more simply for strictly convex
parts of $\gamma$. Then, we prove that the optimal arc $\gamma$ cannot contain any segment.
This allows to completely characterize the optimal arc $\gamma_\alpha$ and its elastic energy $E(\alpha):=E(\gamma_\alpha)$. In the last section we give the proof of the main Theorem. It relies on a sub-additivity property
of the function $E$: $E(\alpha+\beta)\leq E(\alpha)+E(\beta)$ (with a strict inequality if both $\alpha$ and
$\beta$ are positive).

\section{Existence}\label{section2}
We want to solve the minimization problem (\ref{pbpalpha}). We choose to parametrize the arc $\gamma$ whose arc length is denoted by $s$
by the angle $\theta(s)$ between the horizontal line and the tangent. The total length of the arc is $L$ which depends on $\gamma$
and is not fixed. The convexity of the arc is expressed by the fact that $s\mapsto \theta(s)$ is
non decreasing. The curvature is $\theta'(s)$, thus the elastic energy is defined by
$$E(\gamma)= \frac 12 \int_0^L {\theta'(s)}^2\,ds$$
and the regularity we assume is :$\theta$ belongs to the Sobolev space $H^1(0,L)$ which gives a curvature in $L^2$.
Once $\theta(s)$ is known we recover the arc by integrating:
\begin{equation}\label{recgamma}
    x(s)=\int_0^s \cos\theta(u)\,du,\quad\  y(s)=-1+ \int_0^s \sin\theta(u)\,du.
\end{equation}
Therefore, the fact that the arc $\gamma$ ends at $B$ yields
\begin{equation}\label{endgamma}
    \int_0^L \cos\theta(u)\,du=\sin 2\alpha, \quad\  -1+ \int_0^L \sin\theta(u)\,du=-\cos 2\alpha
\end{equation}
and the tangents at $A$ and $B$ impose: $\theta(0)=0$ and $\theta(L)=2\alpha$. To sum up, the minimization problem (\ref{pbpalpha})
can be written in terms of the unknown $\theta$ as
\begin{equation}\label{mintheta}
(\mathcal{P}_\alpha)\qquad
\left\lbrace\begin{array}{c} \vspace{3mm}
  \min\{\frac 12 \int_0^L {\theta'(s)}^2\,ds,\;\theta\in H^1(0,L), \theta'\geq 0 \mbox{ a.e. },\, \\
  \int_0^L \cos\theta(u)\,du=\sin 2\alpha,\,
\int_0^L \sin\theta(u)\,du=-\cos 2\alpha +1,\,\theta(0)=0,\,\theta(L)=2\alpha\}.
\end{array}\right.
\end{equation}
In this section, we are going to prove existence of a minimizer for Problem $(\mathcal{P}_\alpha)$.
We will denote by $\mathcal{M}$ the class of admissible functions $\theta$:
\begin{equation}\label{classM}
\mathcal{M}=\{\theta\in H^1(0,L), \theta'\geq 0 \mbox{ a.e. },\,
  \int_0^L \cos\theta(u)\,du=\sin 2\alpha,\,
\int_0^L \sin\theta(u)\,du=-\cos 2\alpha +1,\,\theta(0)=0,\,\theta(L)=2\alpha\}.
\end{equation}
\subsection{Existence for $\alpha<\pi/2$}
\begin{Theorem}\label{theoexistence1}
Let $\alpha<\pi/2$ be given, there exists a minimizer $\theta_\alpha$ and a corresponding arc $\gamma_\alpha$
for problem $\mathcal{P}_\alpha$.
\end{Theorem}
\begin{proof}
Let $\gamma_n$ parametrized by $\theta_n \in \mathcal{M}$ be a minimizing sequence and let us denote by $L_n$
its length. By convexity the arc $\gamma_n$ is contained in the sector delimited by the two tangents at $A$
and $B$, therefore its length is uniformly bounded: $L_{n} \leq L$.
In order to work on a fixed Sobolev space  $H^1(0,L)$, we assume that $\theta_n$ is formally extended by the constant $\theta _n (L_n)=2\alpha$ on $(L_n, L]$ which does not change the integral
$2E(\gamma_n)=\int_0^L {\theta_n'(s)}^2\,ds$. Since $E(\gamma_n)$ is uniformly bounded, the sequence $\theta_n$
is bounded in $H^1(0,L)$. Up to a subsequence, we can assume that $\theta_n$ converges uniformly on $[0,L]$ to some function $\theta$ and weakly in $H^1(0,L)$. 
Now the lower semi-continuity of $\theta\mapsto \int_0^L {\theta'(s)}^2\,ds$ shows that $\theta$ provides
the desired minimizer.

\medskip
We define the limit curve $\gamma$ in the following way: $L_\gamma =\lim\inf L_{\gamma_n}$ and
$\gamma: [0, L_\gamma]\to \R^2$, $\gamma(s) =\int_0^s e^{i\theta(s)} ds + (0,-1)$. By uniform convergence
$\theta(L_\gamma)=2\alpha$ and $\gamma(L_\gamma)=B$ and the integral constraints of the optimization problem
remain satisfied at the limit also by uniform convergence, which concludes proof of Theorem \ref{theoexistence1}.
\end{proof}
\subsection{Existence for $\alpha=\pi/2$}\label{section2.2}
Now we assume that the arcs $\gamma$ touch the unit disk only at south and north pole $(0,-1)$ and $(0,1)$.
Since the arc has to be included now in the strip $\{-1\leq y\leq 1\}$ (and not in a sector)
the main difficulty is that its length is not {\it a priori} bounded. Therefore our strategy is the following:
\begin{itemize}
\item we solve the minimization problem with a fixed, given, length $L$,
\item we compute the energy of this optimal arc, say $E(L)$,
\item we prove that this energy blows up when $L\to \infty$ or at least converges to something that we control.
\end{itemize}
\begin{Theorem}\label{theoexistence2}
Let $\alpha=\pi/2$, either there exists a minimizer $\theta^*$ and a corresponding arc $\gamma^*$
for problem $\mathcal{P}_{\pi/2}$, or the value of the energy of any arc is larger than $a^2$ where
$a$ is defined in (\ref{optdom2}).
\end{Theorem}
\begin{proof}
{\bf First step:} we fix a length $L>\pi$ and we solve the minimization problem with this length.
Actually, it is exactly the problem of elasticae as introduced by L. Euler since the two extremities points and the
two tangents are given. Existence of an optimal arc follows exactly as in the proof of Theorem \ref{theoexistence1}.
We denote by $\gamma_L$ such an optimal arc and $\theta_L$ its parametrization. Let us remark that we do not need
to assume convexity here.

\medskip\noindent
{\bf Second step:} Let us write the optimality conditions for the optimal arc. We denote it $\theta$ instead of $\theta_L$.
\begin{Proposition}\label{propopti1}
There exists two constants $a_0,c$ with $c^2\geq a_0$ such that that $\theta(s)$ satisfies
$\theta'(s)=\sqrt{c^2-a_0 \sin\theta(s)}$ and $\theta(0)=0$, $\theta(L)=\pi$.
\end{Proposition}
\begin{proof}[Proof of Proposition \ref{propopti1}]
The optimality conditions for the problem of calculus of variations \ref{mintheta} writes:
there exist four Lagrange multipliers $a_0,b,c,d\in \mathbb{R}$ such that for any $v\in H^1(0,L)$:
\begin{equation}\label{op0}
\int_0^L \theta'v'\,ds=a_0\int_0^L \cos\theta(s)\, v(s)\,ds + b\int_0^L \sin\theta(s)\, v(s)\,ds +c v(0) +d v(L)\,.
\end{equation}
This implies that $\theta$ solves the ordinary differential equation
\begin{equation}\label{op1}
\left\lbrace\begin{array}{l}
-\theta''=a_0\cos\theta + b\sin\theta, \quad s\in (0,L)\\
\theta(0)=0,\quad \theta(L)=\pi,\\
\theta'(0)=-c,\quad \theta'(L)=d\,.
\end{array}
\right.
\end{equation}
This equation shows that $\theta$ is a $C^\infty$ function. Actually, this equation is a classical pendulum
equation and its solution can be written explicitly in terms of elliptic functions. We will not  use it here.
Moreover, $\theta$ must satisfy the constraints
$$\int_0^L \cos\theta(s)\,ds=0,\qquad \int_0^L \sin\theta(s)\,ds=2\,.$$
Using these relations together with $v=1$ in (\ref{op0}) yields $c+d=-2b$. Choosing $v=\theta'$ in (\ref{op0})
yields $\frac 12 (c^2-d^2)=2b$. Therefore, we get either $c+d=0$ (and then $b=0$) or $c-d=-2$. Now, the fact
that $\gamma$ has to be externally tangent to the unit disk at points $A$ and $B$ shows that its curvature has to
be smaller than one. This implies $\theta'(0)=-c\leq 1$ and $\theta'(L)=d\leq 1$. Thus $c-d\geq -2$ and if we have
$c-d=-2$, it would imply $c=-1$ and $d=1$. Therefore, in any case we conclude that $c+d=0$ and $b=0$. Multiplying
by $\theta'$ the equation satisfied by $\theta$ and integrating yields
\begin{equation}\label{op2}
{\theta'(s)}^2=c^2-a_0 \sin\theta(s)\,.
\end{equation}
Since $\theta(s)$ has to go from $0$ to $\pi$, equation (\ref{op2}) implies that $c^2\geq a_0$ and $c^2-a_0 \sin\theta$
keeps a constant sign. Therefore, $\theta'$ has to be non negative and is given by
\begin{equation}\label{op3}
\theta'(s)=\sqrt{c^2-a_0 \sin\theta(s)}\,.
\end{equation}
\end{proof}
We come back to the proof of Theorem \ref{theoexistence2}.
Integrating equation (\ref{op2}) between $0$ and $L$ gives the value of the minimal energy:
\begin{equation}\label{ener1}
E(\gamma_L)=E(L)=c^2 \frac{L}{2}\,-a_0\,.
\end{equation}
The constraint $\int_0^L \sin\theta(u)\,du=2$ can be rewritten, using the change of variable $u=\theta(s)$:
\begin{equation}\label{op4}
\int_0^\pi \frac{\sin u}{\sqrt{c^2 - a_0\sin u}}\,du = 2
\end{equation}
or, by symmetry
\begin{equation}\label{op5}
\int_0^{\pi/2} \frac{\sin u}{\sqrt{c^2 - a_0\sin u}}\,du = 1\,.
\end{equation}
We first consider the case $a_0\geq 0$.  In that case
$$\frac{1}{|c|}=\int_0^{\pi/2} \frac{\sin u}{|c|}\,du\leq \int_0^{\pi/2} \frac{\sin u}{\sqrt{c^2 - a_0\sin u}}\,du = 1$$
and then, we have $c^2\geq 1$ and $a_0\leq c^2$. From (\ref{ener1}) we get $E(L)\geq c^2(\frac{L}{2}\,-1)\geq \frac{L}{2}\,-1$.
This shows that for $L\geq L_0:=\frac{\pi}{2} +2$, we have $E(L)\geq \pi/4$ which is the energy of the half circle.
Thus we can restrict us to $L\leq L_0$ and existence follows exactly as in the proof of Theorem \ref{theoexistence1}.

\medskip
Let us now consider the case $a_0<0$. In that case, $E(L)\geq c^2 L/2$. Letting $L$ going to $+\infty$, either $c$
(which depends on $L$) is bounded from below and then $E(L)\to \infty$ which allow us, as in the previous case,
to restrict to some bounded interval $L\leq L_0$ ensuring existence. The only remaining case is when $c\to 0$.
In that case, passing to the limit in equation (\ref{op5}) yields
$$\int_0^{\pi/2} \sqrt{\sin u}\,du = \sqrt{|a_0|}.$$
This implies
$$\liminf E(L)\geq -a_0 = \left( \int_0^{\pi/2} \sqrt{\sin u}\,du\right)^2$$
but the right-hand side is precisely the value found for the minimum of $E$ in Theorem \ref{optdom2}. This finishes the proof.
\end{proof}

\section{Optimality conditions}\label{section3}
In this section, we want first to write the general optimality conditions satisfied by the minimizer $\theta$.
We deduce $W^{2,\infty}$ regularity of the minimizer. Then we prove that the optimal arc is strictly convex: it
contains no segments. At last, we describe the optimal arc explicitly
\subsection{The general optimality conditions}
\begin{Theorem}\label{Ththeta}
 Assume $\theta$ is associated to an optimal arc solution of \eqref{mintheta}.
 Then $\theta\in W^{2,\infty}(0,L)$ and there exist Lagrange multipliers $\lambda_1$, $\lambda_2$
 and a constant $C$ such that, for all $s\in [0,L)$
\begin{equation}\label{first-opt}
\theta'(s) = \left(C-\lambda_1(y(s)+1)-\lambda_2 x(s)\right)^-
\end{equation}
where $(\cdot)^-$ denotes the negative part of a real number.
\end{Theorem}
\begin{proof}
Let us consider the minimizer $\theta(s)$ solution of
\begin{equation}\label{pb_min_theta}
\min_{\theta\in \mathcal{M}} \frac 12\,\int_0^L {\theta'(s)}^2\,ds
\end{equation}
where $\mathcal{M}$ is defined in (\ref{classM}).
Using classical theory for this kind of optimization problem with constraints in a Banach space (see, for instance,
Theorem 3.2 and Theorem 3.3 in \cite{MaurerZowe}),
we can derive the optimality conditions.
More precisely, let us introduce the closed convex cone $K$ of
$L^2(0,L)\times \mathbb{R}^4$ defined by
$$
K:=L^2_+(0,{L})\times \{(0,0,0,0)\},
$$
where
$$
L^2_+(0,{L}):= \left\{\ell \in L^2(0,L) \  ; \
  \ell \geq 0 \right\}.
$$
We also set for $\theta\in W^{1,2}(0,L)$
$$
m(\theta)=\left(\theta', \quad \int_0^{L} \cos(\theta(s)) \ ds-\sin 2\alpha, \quad \int_0^{L}  \sin(\theta(s)) \ ds - 1 +
\cos 2\alpha,
  \quad \theta(0),\quad
\theta({L}) - 2\alpha
\right).
$$
Then, Problem \eqref{pb_min_theta} can be written as
$$
\inf \left\{\frac 12\,\int_0^L {\theta'(s)}^2\,ds, \quad \theta\in W^{1,2}(0,{L} ), \quad m(\theta)\in K\right\}.
$$

As a consequence, for a solution $\theta$ of \eqref{pb_min_theta}
there exist Lagrange multipliers $\ell\in L^2_+(0,{L} )$, $(\lambda_1,\lambda_2,\lambda_3,\lambda_4)\in\mathbb{R}^4$
such that the two following conditions hold:
\begin{eqnarray*}
\int_0^L \theta' v'(s)\,ds=
\langle(\ell,\lambda_1,\lambda_2,\lambda_3,\lambda_4),m'(\theta)(v)\rangle_{L^2(0,L)\times
  \mathbb{R}^4} \quad \forall v\in W^{1,2}(0,L),\\
\left((\ell,\lambda_1,\lambda_2,\lambda_3,\lambda_4),m(\theta)\right)_{L^2(0,{L} )\times
  \mathbb{R}^4} =0.
  \end{eqnarray*}

The two above conditions can be written as
\begin{equation}
 \label{opt1}
   \int_0^{L}  \theta' v' \ ds
= \int_0^{L}   \ell v' \ ds
-\lambda_1 \int_0^{L}  \sin(\theta) v \ ds
+\lambda_2 \int_0^{L}  \cos(\theta) v \ ds
+\lambda_3 v(0) +\lambda_4 v(L)
\end{equation}
\begin{equation}
  \label{opt2}
  \int_0^{L}  \ell \theta' \ ds=0.
\end{equation}
We thus define
\begin{equation}\label{14:41}
f(s)=\lambda_1 \sin(\theta(s))-\lambda_2 \cos(\theta(s))\quad \mbox{for } s \in [0,L]
\end{equation}
and we rewrite \eqref{opt1} as
\begin{equation}
   \int_0^{L}  \theta' v' \ ds
+  \int_0^{L}  f v \ ds
= \int_0^{L}   \ell v' \ ds
\qquad v\in W^{1,2}_0(0,{L} ).\label{opt1bis}
\end{equation}
Let us consider the continuous function $F\in W^{1,\infty}(0,L)$ defined by
\begin{equation}\label{1.1}
F(s):=-\int_0^s f(t) \ dt.
\end{equation}
Then integrating by parts in \eqref{opt1bis} yields (for some constant $C$)
\begin{equation}
 \theta' = -F + \ell-C \quad \text{in} \ (0,{L} ).\label{opt3}
\end{equation}
The above equation implies that
$$
\ell -F-C \geq 0 \quad \text{in} \ (0,{L} ).
$$
On the other hand condition \eqref{opt2} yields
$\ell \theta' =0 \quad \text{in} \ (0,{L} )$ which implies
$\ell(\ell-F-C) =  0 \quad \text{in} \ (0,{L} )$, thanks to relation \eqref{opt3}.

We rewrite the above equality by using the decomposition
$F+C=g^+-g^-$, (where $g^+$ and $g^-$ are the positive and negative parts of $F+C$):
\begin{eqnarray*}
\ell(\ell-F-C)=(\ell-g^++g^+)(\ell-g^++g^-) = (\ell-g^+)^2 +
g^-(\ell-g^+)+\\
+g^+(\ell-g^+)+g^+g^-
=(\ell-g^+)^2 + g^- \ell + g^+ (\ell-F-C)
\end{eqnarray*}
which is the sum of three non-negative terms. Thus
\begin{equation}\label{1.2}
\ell=(F+C)^+
\end{equation}
and in particular, from \eqref{opt3},
\begin{equation}\label{15:18}
 \theta' = (F+C)^- \quad \text{in} \ (0,{L} )
\end{equation}
We deduce that $\theta\in W^{2,\infty}(0,L)$.

\medskip
Using \eqref{recgamma}, we can write
\begin{equation}\label{1.3}
F(s)=-\int_0^s (\lambda_1 \sin \theta(s) -\lambda_2 \cos \theta(s)\,ds= -\lambda_1(y(s)+1) +\lambda_2 x(s)
\end{equation}
The above relation and \eqref{15:18} yield
\eqref{first-opt}.
\end{proof}

\subsection{Optimality conditions on a strictly convex arc}
On a strictly convex arc $\gamma$ where $\theta'>0$ we deduce from the previous section that $\ell=0$ and from \eqref{opt1bis} we see 
that $\theta$ satisfies the ordinary differential equation
\begin{equation}\label{eqdiff1}
-\theta''= \lambda_2 \cos\theta -\lambda_1 \sin \theta\,.
\end{equation}
Multiplying by $\theta'$ and integrating between any point $C$ of the arc with parameter $s_C$ and $s$ yields
\begin{equation}\label{eqdiff2}
\frac{1}{2}\,{\theta'(s)}^2=\frac{1}{2}\,{\theta'(s_C)}^2 - \lambda_2 (\sin\theta(s)-\sin\theta(s_C)) -\lambda_1(\cos \theta(s)-\cos\theta(s_C))\,.
\end{equation}
On the other hand, it is a classical result in differential geometry that the shape derivative of the 
functional $\frac{1}{2}\,\int_\gamma k^2(s)\,ds$ is $-\int_\gamma (k''+ \frac{1}{2}\,k^3) V.n \,ds$,
see for example the Appendix in \cite{BHT}. Therefore, for an optimal arc, we have $k''+ \frac{1}{2}\,k^3=0$.
Differentiating once \eqref{eqdiff1}, we get
$$-\theta'''=-k''= (-\lambda_2 \sin\theta -\lambda_1 \cos \theta) k = \frac{1}{2}\,k^3\,.$$
Therefore, when $k>0$ and comparing with \eqref{eqdiff2}, we finally get 
\begin{Proposition}
On a strictly convex arc, the optimality condition reads
\begin{equation}\label{eqdiff3}
\frac{1}{2}\,{\theta'(s)}^2= - \lambda_2 \sin\theta(s) - \lambda_1 \cos \theta(s))\,.
\end{equation}
\end{Proposition}
\section{Proof of the main theorem}
\subsection{There are no segments on the boundary}
Let us assume that there is a segment on the boundary of the optimal arc, starting at some point $M=(x_M,y_M)$
finishing at some point $N=(x_N,y_N)$ and let us denote by $\beta=\theta(s_M)=\theta(s_N)$ the angle between
the horizontal axis and this segment ($s_M$ denotes the curvilinear abscissa at point $M$). 

According to Theorem \ref{Ththeta}, the function $\theta(s)$ is globally $C^1$, therefore, $\theta'$ vanishes at point $M$: $\theta'(s_M)=0$. Inserting in \eqref{eqdiff3} yields
$\lambda_1=-\lambda_2 \tan \beta$ (the case $\beta=\pi/2$ is treated exactly in the same way) and 
\begin{equation}\label{eqdiff4}
{\theta'}^2(s)= \frac{2\lambda_2}{\cos \beta}\,\sin (\beta - \theta)\ \ s\in [0,s_M].
\end{equation}
Using such a relation, it can be seen that $\theta'$ cannot vanish at two extremities of a strictly convex arc.
In other words, if there is a segment on the optimal arc, this one is unique.

\medskip
Exactly in the same way, but starting at the point $B=(\sin 2\alpha, \cos 2\alpha)$ instead of $A$, we can see that
the curvature $\theta'$ vanishes at point $N$: $\theta'(s_N)=0$. Inserting in \eqref{eqdiff3} yields
$\lambda_1=-\lambda_2 \tan \beta$ and we also have
\begin{equation}\label{eqdiff5}
{\theta'}^2(s)= \frac{2\lambda_2}{\cos \beta}\,\sin (\beta - \theta)\ \ s\in [s_N,L].
\end{equation}
Now, by convexity, $\theta>\beta$ on $(s_N,L)$ while $\theta < \beta$ on $(0,s_M)$ then $\sin (\beta - \theta)$ have opposite
signs on the intervals $(0,s_M)$ and $(s_N,L)$ which is not possible.

\subsection{Expression of the optimal energy}
Since we know now that there are no segments on the optimal arc, the optimality condition \eqref{eqdiff1} holds on the whole arc.
We follow a similar approach as Section \ref{section2.2}.
Integrating this relation between $0$ and $L$ yields
\begin{equation}\label{eqdiff6}
\theta'(0)-\theta'(L) = \lambda_2 \sin 2\alpha - \lambda_1 (1-\cos 2\alpha)\,.
\end{equation}
In the same way, multiplying \eqref{eqdiff1} by $\theta'$ and integrating between $0$ and $L$ yields
\begin{equation}\label{eqdiff7}
\frac{1}{2}\,\left({\theta'(L)}^2 -{\theta'(0)}^2 \right) = - \lambda_2 (\sin 2\alpha ) -\lambda_1(\cos 2\alpha -1 )\,.
\end{equation}
Adding \eqref{eqdiff6} and \eqref{eqdiff7}, we get
\begin{equation}\label{eqdiff8}
\left({\theta'(L)} -{\theta'(0)} \right) \left(\frac{1}{2}\, ({\theta'(L)} + {\theta'(0)}) -1 \right) = 0 \,.
\end{equation}
Now, since the curve is externally tangent to the disk at the points $A$ and $B$, the curvature at these points is less than one.
Therefore ${\theta'(L)} + {\theta'(0)} \leq 2$ and we infer, in all cases that ${\theta'(L)} = {\theta'(0)}$. This implies, according
to \eqref{eqdiff6} that
\begin{equation}\label{eqdiff8bis}
\lambda_2 \sin 2\alpha - \lambda_1 (1-\cos 2\alpha) =0
\end{equation}
or $\lambda_2 = \lambda_1 \tan \alpha$. Coming back to \eqref{eqdiff3}, this yields
\begin{equation}\label{eqdiff9}
\frac{1}{2}\,{\theta'(s)}^2= - \frac{\lambda_1}{\cos\alpha} \cos (\theta(s) - \alpha)\,.
\end{equation}
To determine the other Lagrange multiplier $\lambda_1$, we use the constraint $\int_0^L \cos \theta(s)\, ds =\sin 2\alpha$ where we make
the change of variable $u=\theta(s)$. This gives
$$\sqrt{\frac{\cos \alpha}{-2\lambda_1}}\,\int_0^{2\alpha} \frac{\cos u\,du}{\sqrt{\cos(u-\alpha)}}\,=\,\sin 2\alpha$$
or
$$-2\lambda_1 = \frac{\cos \alpha}{\sin^2 2\alpha}\, \left(\int_0^{2\alpha} \frac{\cos u\,du}{\sqrt{\cos(u-\alpha)}} \right)^2\,.$$
Now
$$\int_0^{2\alpha} \frac{\cos u\,du}{\sqrt{\cos(u-\alpha)}} = \int_{-\alpha}^{\alpha} \frac{\cos(\alpha +t)\,dt}{\sqrt{\cos(t)}} =
2\cos \alpha \int_{0}^{\alpha} \sqrt{\cos t}\,dt.$$
Therefore
$$-\lambda_1 = \frac{\cos \alpha}{2 \sin^2 \alpha}\,\left(\int_{0}^{\alpha} \sqrt{\cos t}\,dt\right)^2.$$
Replacing in \eqref{eqdiff9} yields
\begin{equation}\label{eqdiff10}
\frac{1}{2}\,{\theta'(s)}^2= \frac{1}{2 \sin^2 \alpha}\,\left(\int_{0}^{\alpha} \sqrt{\cos t}\,dt\right)^2 \,\cos (\theta(s) - \alpha)\,.
\end{equation}
Integrating \eqref{eqdiff10} between $0$ and $L$ gives the value of the energy for the optimal arc, we denote it by $E(\alpha)$:
\begin{equation}\label{energy}
E(\alpha)=\frac{1}{2}\,\int_{0}^L {\theta'(s)}^2= \frac{1}{\sin \alpha}\,\left(\int_{0}^{\alpha} \sqrt{\cos t}\,dt\right)^2 \,.
\end{equation}
\begin{Remark}
Using Cauchy-Schwarz inequality, we can see that the energy for the optimal arc is indeed better than the energy of the corresponding arc of circle:
$$E(\alpha) \leq \frac{1}{\sin \alpha} \,\int_{0}^{\alpha} \cos t\,dt \int_{0}^{\alpha} 1\,dt= \alpha = E(\mbox{arc of circle}).$$
\end{Remark}

\subsection{Study of the function $\alpha\mapsto E(\alpha)$}
Here we want to study the function $E(\alpha)$ defined in \eqref{energy}, for $\alpha \in (0,\pi/2]$. We prove
\begin{Proposition}
The function $\alpha \mapsto E(\alpha)$ is concave and sub-additive
\begin{equation}\label{subad}
\forall \alpha, \beta \in (0,\pi/2],\quad E(\alpha + \beta)< E(\alpha) + E(\beta)\,.
\end{equation}
\end{Proposition}
\begin{proof}
The first derivative of $E(\alpha)$ can be written
$$E'(\alpha)=h(\alpha) (2-h(\alpha))\quad \mbox{where}\quad h(\alpha)=\frac{\sqrt{\cos \alpha}}{\sin \alpha}\, \int_{0}^{\alpha} \sqrt{\cos t}\,dt.$$
Note that, by Cauchy-Schwarz inequality,
$$\int_{0}^{\alpha} \sqrt{\cos t}\,dt \leq \sqrt{\alpha \sin \alpha}\  \Longrightarrow \ h(\alpha) \leq 
\sqrt{\frac{\alpha \cos\alpha}{\sin \alpha}}\,\leq 1\,.$$
Now the second derivative of $E(\alpha)$ is
$$E''(\alpha)=\frac{2\cos\alpha}{\sin\alpha}\,(1-h(\alpha))\left(1-h(\alpha)(1+\tan^2(\alpha)/2)\right).$$
Therefore, to prove that $E$ is concave on $(0,\pi/2]$, it suffices to prove that 
$$h(\alpha) \geq \frac{1}{1+\frac{\tan^2(\alpha)}{2}}\ \Longleftrightarrow \int_{0}^{\alpha} \sqrt{\cos t}\,dt
\geq \frac{2 \sin \alpha}{\sqrt{\cos \alpha}(2+\tan^2(\alpha))}.$$
Now the difference 
$$R(\alpha)= \int_{0}^{\alpha} \sqrt{\cos t}\,dt - \frac{2 \sin \alpha}{\sqrt{\cos \alpha}(2+\tan^2(\alpha))}$$
satisfies 
$$R'(\alpha)=\frac{16 \sin^2(\alpha) \sqrt{\cos\alpha}}{(\cos (2\alpha) +3)^2} \,\geq 0\;,\quad R(0)=0$$
and then $R(\alpha) \geq 0$ (and $E''(\alpha)<0$ if $\alpha >0$) which proves the result and the (strict) concavity of $E$.

For the sub-additivity, we just write, since $E(0)=0$:
$$E(\alpha + \beta) - E(\alpha) - E(\beta) = \int_0^\beta \int_v^{v+\alpha} E''(t)\,dt\,dv <0\,.$$
\end{proof}
\subsection{Conclusion}
First we use the strict sub-additivity of the function $E$ to claim that if
we have several contact points in a range less than $\pi/2$, it is better to take only the two extremities.
Therefore, there are only two possibilities:
\begin{itemize}
\item either there are only two contact points, which are diametrically opposite
\item or there are three contact points with angles (with the previous notations) $\alpha,\beta,\pi -\alpha-\beta$
satisfying $0<\alpha\leq \pi/2, \;0<\beta\leq \pi/2, \pi/2 \leq \alpha + \beta \leq \pi$. 
\end{itemize} 
We are going to prove that we have necessarily equality $\alpha=\pi/2$ or $\beta=\pi/2$ or $\alpha + \beta=\pi/2$
which implies, thanks to the sub-additivity, that we have better to take only two points and we are led to the first case.
Let us fix $\gamma=\alpha + \beta$ and choose the optimal way to split the angle $\gamma$ in two parts. For that
purpose, let us introduce the function $e(t)=E(t)+E(\gamma -t)$. To satisfy the constraints, we have to choose $t$ between
$\gamma - \pi/2$ and $\pi/2$. The derivative of $e(t)$ is $e'(t)=E'(t)-E'(\gamma -t)$. Now by concavity of the function $E$,
we see that $e(t)$ is increasing when $t\in [\gamma -\pi/2, \gamma/2]$ and decreasing when $t\in [\gamma/2, \pi/2]$.
Therefore, $e(t)$ reaches its minimum when $t=\gamma -\pi/2$ or $t=\pi/2$. This corresponds to the equality cases
previously mentioned and that shows that the optimal configuration corresponds to only two contact points diametrically
opposite.

\medskip
Let us now give the explicit expression of the optimal arc when $\alpha=\pi/2$.
According to \eqref{eqdiff8bis},
it implies that $\lambda_1=0$ and using \eqref{eqdiff3} we see that the function $\theta$ is solution of the ordinary 
differential equation ${\theta'}^2=-2\lambda_2 \sin \theta$. We determine $\lambda_2$ by writing
$\int_0^L \sin \theta(s)\,ds=2$ which gives
$$-2\lambda_2=\left(\int_0^{\pi/2} \sqrt{\sin u}\,du\right)^2\,.$$
Let us denote by $a$ the integral
$$a=\int_0^{\pi/2} \sqrt{\sin u}\,du = \int_0^{\pi/2} \sqrt{\cos u}\,du$$
thus, finally $\theta$ is solution of
\begin{equation}\label{eqdifinal}
\theta'=\frac{1}{a}\,\sqrt{\sin \theta},\ \theta(0)=0\,.
\end{equation}
The optimal curve is now given by
\begin{equation}\label{optcurve}
\left\lbrace\begin{array}{l}
x(s)=\int_0^s \cos\theta(t)\,dt=\frac{1}{a}\;\int_0^{\theta(s)}\frac{\cos u}{\sqrt{\sin u}}\, du=
\frac{2}{a}\,\sqrt{\cos \theta(s)} \\
y(s)=-1+\int_0^s \sin\theta(t)\,dt=-1+\frac{1}{a}\;\int_0^{\theta(s)}\frac{\sin u}{\sqrt{\sin u}}\, du=-1+
\frac{1}{a}\;\int_0^{\theta(s)} \sqrt{\sin u}\, du\,.
\end{array}\right.
\end{equation}
It remains to find the expression of the length of the arc. The semi-length $L$ corresponds to the value $\theta(L)=\pi$.
Now the exact solution of the differential equation \eqref{eqdifinal} is $\theta(s)=2 am(\frac{as}{2}\,+b\,|\,2) +\frac{\pi}{2}$
where $am$ is the {\it Jacobi amplitude}, inverse of the elliptic integral of the first kind $F$, see \cite{AbSt}
and $b=F(-\frac{\pi}{4}|2)$.
Therefore, $L$ must satisfy 
$$\frac{\pi}{2} = \theta\left(\frac{L}{2}\,\right)=2 am(\frac{aL}{4}\,+b|2) +\frac{\pi}{2}$$
or $am(\frac{aL}{4}\,+b|2) = 0$. Therefore, $L=-4b/a$. Now
$$b=F(-\frac{\pi}{4}|2)=\frac{1}{2}\,\int_0^{-\pi/2} \frac{dt}{\sqrt{\cos t}}=-
\frac{1}{2}\,\int_0^{\pi/2} \frac{dt}{\sqrt{\cos t}}$$
and finally
\begin{equation}\label{lengthL}
L=2 \frac{\int_0^{\pi/2} \frac{dt}{\sqrt{\cos t}}}{\int_0^{\pi/2} \sqrt{\cos t}\,dt}\,.
\end{equation}
The total length of the curve is twice this value. At last, the elastic energy of the optimal domain is
$$E(\Omega^*)=2 E\left(\frac{\pi}{2}\right) = 2 \left(\int_0^{\pi/2} \sqrt{\cos t}\,dt\right)^2\,.$$
One more time, Cauchy-Schwarz inequality shows that this value is less than $\pi$, elastic energy of the unit circle.
The numerical value is 2.8711.

\section*{Acknowledgement}

This paper has been realized while Othmane Mounjid was Master student at \'Ecole des Mines de Nancy - Universit\'e de Lorraine
involved in a research project "Parcours Recherche" with A. Henrot.

The work of Antoine Henrot is supported by the project ANR-12-BS01-0007-01-OPTIFORM {\it Optimisation de formes} 
financed by the French Agence Nationale de la Recherche (ANR).


\end{document}